\documentclass{amsart}
\usepackage{amssymb,amsmath, amsfonts, amsrefs, tikz, epsfig, float}
\usepackage{amsthm}
\usepackage{mathrsfs}
\usepackage{eucal}
\usepackage{enumitem, indentfirst}
\usepackage[fleqn,tbtags]{mathtools}
\usepackage{color}

\usepackage{multicol}
 \newtheorem{theorem}{Theorem}[section]%
 \newtheorem{corollary}[theorem]{Corollary}

 \theoremstyle{definition}

 \theoremstyle{remark}
  \numberwithin{equation}{section}

\renewcommand{\phi}{\varphi}
\renewcommand{\theta}{\vartheta}

\DeclareMathOperator{\tform}{\mathfrak{t}}

\DeclareMathOperator{\wform}{\mathfrak{w}}

\DeclarePairedDelimiterX\sipt[2]{(}{)_{\tform}}{#1\,\delimsize\vert\,#2}
\DeclarePairedDelimiterX\sipv[2]{(}{)_{v}}{#1\,\delimsize\vert\,#2}
\DeclarePairedDelimiterX\sipw[2]{(}{)_{w}}{#1\,\delimsize\vert\,#2}

\newcommand{\alg}{\mathscr{A}}

\newcommand{\ideal}{\mathscr{I}}

\newcommand{\dupC}{\mathbb{C}}

\newcommand{\dom}{\operatorname{dom}}

\newcommand{\ran}{\operatorname{ran}}

\newcommand{\lef}{\mathscr{L}(E,F)}

\newcommand{\bha}{\mathscr{B}(\hila)}

\newcommand{\sigef}{\sigma(E,F)}
\newcommand{\sigfe}{\sigma(F,E)}
\newcommand{\hil}{H}

\newcommand{\hila}{H_A}

\newcommand{\kil}{K}
\newcommand{\bh}{\mathscr{B}(\hil)}
\newcommand{\bhk}{\mathscr{B}(\hil,\kil)}

\DeclarePairedDelimiterX\abs[1]{\lvert}{\rvert}{#1}
\DeclarePairedDelimiterX\sip[2]{(}{)}{#1\,\delimsize\vert\,#2}
\DeclarePairedDelimiterX\sipta[2]{(}{)_{\!_{\widetilde{A}}}}{#1\,\delimsize\vert\,#2}
\DeclarePairedDelimiterX\sipf[2]{(}{)_{f}}{#1\,\delimsize\vert\,#2}
\DeclarePairedDelimiterX\sipg[2]{(}{)_{g}}{#1\,\delimsize\vert\,#2}
\DeclarePairedDelimiterX\siptw[2]{(}{)_{\tform+\wform}}{#1\,\delimsize\vert\,#2}
\DeclarePairedDelimiterX\set[2]{\{}{\}}{#1\,\delimsize\vert\,#2}
\DeclarePairedDelimiterX\dual[2]{\langle}{\rangle}{#1,#2}
\DeclarePairedDelimiterX\duale[2]{\langle}{\rangle_1}{#1,#2}
\DeclarePairedDelimiterX\dualk[2]{\langle}{\rangle_2}{#1,#2}
\DeclarePairedDelimiterX\sdual[2]{[}{]}{#1,#2}
\DeclarePairedDelimiterX\sipa[2]{(}{)_{\!_A}}{#1\,\delimsize\vert\,#2}
\DeclarePairedDelimiterX\sipc[2]{(}{)_{\!_C}}{#1\,\delimsize\vert\,#2}
\DeclarePairedDelimiterX\sipab[2]{(}{)_{\!_{A+B}}}{#1\,\delimsize\vert\,#2}
\DeclarePairedDelimiterX\sipb[2]{(}{)_{\!_B}}{#1\,\delimsize\vert\,#2}
\newcommand{\anti}[1]{\bar{#1}'}

\newcommand{\opmatrix}[4]{\left[\begin{array}{cc}\!\! #1 &  #2\!\!\\ \!\! #3& #4\!\!\end{array}\right]}
\newcommand{\kismatrix}[4]{\begin{bsmallmatrix} #1 &  #2\\ #3& #4\end{bsmallmatrix}}

\allowdisplaybreaks
\begin{document}
\sloppy 
\title[Self-adjoint extensions and the Strong Parrott Theorem]{Operators on anti-dual pairs: self-adjoint extensions and the Strong Parrott Theorem}

\author[Zs. Tarcsay]{Zsigmond Tarcsay 
}
\thanks{The corresponding author Zs. Tarcsay was supported by DAAD-TEMPUS Cooperation Project ``Harmonic Analysis and Extremal Problems'' (grant no. 308015). Project no. ED 18-1-2019-0030 (Application-specific highly reliable IT
solutions)
has been implemented with the support provided from the National Research,
Development and Innovation Fund of Hungary, financed under the Thematic
Excellence
Programme funding scheme.}

\address{%
Zs. Tarcsay \\ Department of Applied Analysis  and Computational Mathematics\\ E\"otv\"os Lor\'and University\\ P\'azm\'any P\'eter s\'et\'any 1/c.\\ Budapest H-1117\\ Hungary}

\email{tarcsay@cs.elte.hu}

\author[T. Titkos]{Tam\'as Titkos}
\thanks{T. Titkos was supported by the Hungarian National Research, Development and Innovation Office - NKFIH (Grant No. PD128374 and Grant No. K115383), by the J\'anos Bolyai Research Scholarship of the Hungarian Academy of Sciences, and by the \'UNKP-18-4-BGE-3 New National Excellence Program of the Ministry of Human Capacities.}

\address{T. Titkos \\ Alfr\'ed R\'enyi Institute of Mathematics\\ Re\'altanoda utca 13-15.\\ Budapest H-1053\\ Hungary\\ and BBS University of Applied Sciences\\ Alkotm\'any u. 9.\\ Budapest H-1054\\ Hungary }
\email{titkos@renyi.hu}



\maketitle
\begin{abstract}
The aim of this paper is to develop an approach to obtain self-adjoint extensions of symmetric operators acting on anti-dual pairs. The main advantage of such a result is that it can be applied for structures not carrying a Hilbert space structure or a normable topology. In fact, we will show how hermitian extensions of linear functionals of involutive algebras can be governed by means of their induced operators. As an operator theoretic application, we provide a direct generalization of Parrott's theorem  on contractive completion of $2$ by $2$ block operator-valued matrices. To exhibit the applicability in noncommutative integration, we characterize hermitian extendibility of symmetric functionals defined on a left ideal of a $C^{*}$-algebra.

\end{abstract}

\section*{Introduction}
The question whether a self-adjoint extension exists arises naturally in various situations when a partially defined (bounded or unbounded) symmetric operator is given. For classical results we refer the reader to \cites{Ando-Nishio,Coddington-deSnoo,HassiMalamudSnoo,Sebestyen93} and the references therein, for more recent results see for example \cites{BaidukHassi,MMM1}. In our previous paper \cite{KV}, we have developed a Krein--von Neumann type extension theory for positive operators acting on anti-dual pairs. That technique is general enough to not only  overcome the lack of a Hilbert space structure, but also the lack of a normable topology. Our running example in \cite{KV} -- illustrating the applicability of the general setting -- came from noncommutative integration theory. Namely, we have demonstrated how functional extensions can be governed by their induced operators. The aim of the present paper is to continue these investigations and to discuss the problem of self-adjoint extendibility. \par
Below we briefly describe the content. 
Section \ref{S: Preliminaries} contains a short overview of concepts and earlier results that help the reader to follow the proofs. In particular, we briefly sketch the construction of the generalized Krein--von Neumann extension that serves as the basis for our treatment. In Section \ref{S: Symmetric} we examine the extension problem in our anti-dual pair setting. The main result Theorem \ref{T:symmetricext}   generalizes Krein's theorem on the existence of a norm preserving self-adjoint extension of a bounded symmetric operator \cite{Krein}*{Theorem 5.33}. Due to the lack of norm, we are going to consider extensions bounded by a fixed positive operator $A$. It will turn out that extensions preserving the $A$-bound form an operator interval. As a nice application of Theorem \ref{T:symmetricext}, in Section \ref{S:Parrott} we will generalize Yamada's recent result \cite{Yamada}, which is an extension of the Strong Parrott Theorem \cites{F-T,parrott}. We will close the paper by demonstrating that Theorem \ref{T:symmetricext} on self-adjoint extensions is an effective generalization. Namely, we shall see in Section \ref{S: funct} how this result can be applied  to obtain hermitian extensions of functionals of an involutive algebra.

\section{Preliminaries}\label{S: Preliminaries}
In this section we summarize shortly all the notions and tools to make the presentation self-contained. For more details we refer the reader to \cite{KV}*{Section 2 and 3}.
An anti-dual pair denoted by $\dual FE$ is a system of two complex vector spaces $E$ and $F$ intertwined by a separating sesquilinear map $\dual\cdot\cdot:F\times E\to\dupC$, i.e., $\dual\cdot\cdot$ is linear in its first, and conjugate linear in its second argument. Let $D$ be a linear subspace of $E$. We call a linear operator $A:D\to F$  \emph{symmetric}, if $\dual{Ax}{y}=\overline{\dual{Ay}{x}}$ holds for all $ x,y\in D$. In analog with the Hilbert space case, an operator $A:D\to F$ is said to be \emph{positive}, if
\begin{equation*}
    \dual{Ax}{x}\geq 0,\qquad x\in D.
\end{equation*}

Just as in the dual pair case, we endow $E$ and $F$ with the corresponding weak topologies $\sigef$, resp. $\sigfe$, induced by the families $\set{\dual f\cdot }{f\in F}$, resp. $\set{\dual\cdot x}{x\in E}$. Both $\sigef$ and $\sigfe$ are locally convex Hausdorff topologies with duality properties 
\begin{equation}\label{E:E'=F}
    \anti E=F\qquad\mbox{and}\qquad F'=E.
\end{equation}
Here $F'$ and $\anti E$ refer to the topological dual and anti-dual space of $F$ and $E$, respectively, and  the vectors  $f\in F$ and $x\in E$ are identified  with $\dual f\cdot$, and  $\dual \cdot x$, respectively.   We will call the anti-dual pair $\dual FE$ weak-* sequentially complete if the topological vector space $(F,\sigfe)$ is sequentially complete.

One of the useful properties of weak topologies is the following: for a topological vector space $(V,\tau)$, a linear operator $T:V\to F$ is continuous with respect to $\tau$ and $\sigfe$ if and only if the linear functionals
\begin{equation*}
    T_x(v)\coloneqq \dual{Tv}{x},\qquad v\in V,
\end{equation*}
are continuous for every $x\in E$. 

This fact and \eqref{E:E'=F} enables us to define the adjoint (that is, the topological transpose) of a weakly continuous operator. Let $\dual{F_1}{E_1}$ and $\dual{F_2}{E_2}$ be anti-dual pairs and  $T:E_1\to F_2$ a weakly continuous linear operator. Then the weakly continuous linear operator $T^*:E_2\to F_1$ satisfying \begin{equation*}
    \dual{Tx_1}{x_2}_2=\overline{\dual{T^*x_2}{x_1}_1},\qquad  x_1\in E_1,x_2\in E_2,
\end{equation*} 
is called the adjoint of $T$. The set of \emph{everywhere defined} weakly continuous (i.e., $\sigef$-$\sigfe$ continuous) linear operators $T:E\to F$ will be denoted by $\lef$. In the case when $\hil$ and $\kil$ are Hilbert spaces, $\mathscr{L}(\hil,\kil)$ coincides with the set $\bhk$ of bounded linear operators from $\hil$ to $\kil$. An operator $T\in \lef$ is called self-adjoint if $T^*=T$. Obviously, everywhere defined symmetric operators (and hence everywhere defined positive operators) are automatically weakly continuous and self-adjoint. 
\par
Now we proceed by recalling the construction of the Krein-von Neumann extension of a positive operator. We will use the notations of this section without further notice. For more details see \cite{KV}*{Theorem 3.1 $(iv)\Rightarrow(i)$}.
Let $\dual{F}{E}$ be a $w^*$-sequentially complete anti-dual pair and let $A:E\supseteq\dom A\to F$ be a positive operator with domain $\dom A$. Assume further that
for any $y$ in $E$ there is $M_y\geq0$ such that 
\begin{equation}\label{E:M_y}
 \abs{\dual{Ax}{y}}^2\leq M_y\dual{Ax}{x}\qquad \textrm{for all $x\in\dom A$.}
\end{equation}
This assumption guarantees that one can build a Hilbert space $\hil_{A}$ by taking the Hilbert space completion of the inner product space $\big(\ran A,\sipa{\cdot}{\cdot}\big)$, where
\begin{equation}\label{E:sipa}
\sipa{{A}x}{{A}x'}:=\dual{{A}x}{x'},\qquad x,x'\in\dom {A}.
\end{equation}
Again, by \eqref{E:M_y}, the canonical embedding operator 
\begin{equation}\label{E:J0}
    J_0:H_A\supseteq\ran A\to F,\qquad J_0(Ax):={A}x
\end{equation}
is weakly continuous, and thus admits a unique continuous extension $J$ to $\hila$ by $w^*$-sequentially completeness of $F$. Since $J\in\mathscr{L}(\hil_{A} ;F)$, we have $J^*\in\mathscr{L}(E ;\hil_{A})$. From 
\begin{equation*}
    \sipa{Ax'}{J^*x}=\dual{J(Ax')}{x}=\dual{Ax'}{x}=\sipa{Ax'}{Ax}, \qquad x,x'\in\dom A,
\end{equation*} 
it follows that $J^*x={A}x$ for all $x\in\dom A$. As for any $x\in\dom {A}$ we have $$JJ^*x=J({A}x)={A}x,$$ 
the operator $JJ^*\in\lef$ is a positive extension of $A$. We will refer to $A_N:=JJ^*$ as the Krein-von Neumann extension of $A$.
\par
We remark that the extension result above is closely related to the theory of reproducing kernels (see for example \cites{LS, pv, Yamada}). In fact, one can say that the operator $A:E\supseteq D\to F$ is a restriction of a reproducing kernel if and only if \eqref{E:M_y} holds. Finally we mention that our assumption on $F$ to be $w^*$-sequentially complete is weaker than quasi-completeness imposed by Schwartz in \cite{LS}.

\section{Self-adjoint extensions of symmetric operators}
\label{S: Symmetric}

M.G. Krein proved in \cite{Krein} that every bounded symmetric Hilbert space operator possesses a norm preserving self-adjoint extension. The problem of constructing self-adjoint extensions of a symmetric operator arises in our anti-dual pair setting naturally. Since we cannot speak about norm preservation due to the lack of norm, we need to find a suitable notion to generalize Krein's theorem. Observe that the norm of a self-adjoint operator $S\in\bh$ can be expressed by means of the partial order induced by positivity. Namely, $\|S\|$ is the smallest constant $\alpha\geq0$ such that $-\alpha I\leq S\leq \alpha I$. Based on this observation, a symmetric operator $S_0:E\supseteq\dom S_0\to F$ is called $A$-\emph{bounded} for a fixed positive operator $A\in\lef$ if 
\begin{equation}\label{E:symmetricext}
\abs{\dual{S_0x}{y}}^2\leq \alpha^2\cdot \dual{Ax}{x}\dual{Ay}{y},\qquad x\in \dom S_0, y\in E,
\end{equation}
holds. The smallest constant $\alpha$ is called the $A$-\emph{bound} of $S_0$ and is denoted by $\alpha_A(S_0)$.
We will call the extension $S\supset S_0$ $A$-\emph{bound preserving} if $\alpha_A(S_0)=\alpha_A(S)$. 
\par
In the next theorem, which is the main result of this section, we will present a sufficient condition to guarantee for a symmetric linear operator that it possesses a self-adjoint extension. Moreover, we describe the set of all $A$-bound preserving extensions of a given symmetric operator.

\begin{theorem}\label{T:symmetricext}
Let $\dual FE$ be a weak-$^*$ sequentially complete anti-dual pair and let    $S_0:\dom S_0\to F$ be a symmetric operator, i.e.,
\begin{align*}
\dual{S_0x}{y}=\overline{\dual{S_0y}{x}},\qquad x,y\in \dom S_0.
\end{align*}
Suppose that  $S_0$ is $A$-bounded with some positive operator $A\in\lef$. Then there exist two distinguished self-adjoint extensions $S_m, S_M\in \lef$ of $S_0$ such that
\begin{equation*}
\alpha_A(S_m)=\alpha_A(S_M)=\alpha_A(S_0).    
\end{equation*}
In fact, the interval $[S_m, S_M]$ consists exactly of all self-adjoint extensions $S\supset S_0$ such that $\alpha_A(S)=\alpha_A(S_0)$:
\begin{align}\label{E:intervalSmSM}
    [S_m, S_M]=\set{S\in\lef}{S_0\subset S=S^*,~\alpha_A(S)=\alpha_A(S_0)}.
\end{align}
\end{theorem}
\begin{proof}
Introduce the following linear manifold
\begin{align*}
\dom \widehat{S}_0:=\set{Ax}{x\in \dom S_0}\subseteq \hila,
\end{align*}
and  fix an $x\in \dom S_0$. Let us define the conjugate linear functional $f_x$ as 
\begin{align*}
f_x: \ran A\subseteq \hila\to \dupC,\qquad f_x(Ay):=\dual{S_0x}{y},\qquad y\in E.
\end{align*}
Observe that $f_x$ is continuous, because
\begin{align}\label{E:fxnorm}
\abs{f_x(Ay)}^2\leq \alpha^2\cdot \dual{Ax}{x},\qquad y\in E, \sipa{Ay}{Ay}\leq 1.
\end{align} 
holds for some $\alpha$ by $A$-boundedness. According to the Riesz representation theorem, there exists a unique representing vector $\zeta_x\in\hila$ such that  
\begin{align*}
\dual{S_0x}{y}=\sipa{\zeta_x}{Ay}, \qquad y\in E.
\end{align*}
If $Ax=Ax'$ for some $x,x'$ in $\dom S_0$ then $f_x=f_{x'}$ and thus $\zeta_x=\zeta_{x'}$. Therefore the mapping $\widehat{S}_0:\dom \widehat{S}_0\to \hila$, $$\widehat{S}_0(Ax):=\zeta_x$$ is well defined and linear. By \eqref{E:fxnorm} we have
\begin{align*}
\sipa{\widehat{S}_0(Ax)}{\widehat{S}_0(Ax)}\leq \alpha^2\cdot \sipa{Ax}{Ax},\qquad x\in \dom S_0,
\end{align*}
whence we infer that $\widehat{S}_0$ is bounded with $\|\widehat{S}_0\|= \alpha_A(S_0)$. For $x,y$ in $\dom S_0$, 
\begin{align*}
\sipa{\widehat{S}_0(Ax)}{Ay}=\dual{S_0x}{y}=\overline{\dual{S_0y}{x}}=\sipa{Ax}{\widehat{S}_0(Ay)},
\end{align*}
hence $\widehat{S}_0$ is symmetric. Introduce the operators 
 \begin{align}
     \widehat{T}_m\coloneqq\|\widehat{S}_0\|+ \widehat{S}_0, \qquad \widehat{T}_M\coloneqq \|\widehat{S}_0\|-\widehat{S}_0.
 \end{align}
 Clearly, $\widehat{T}_m$ and $\widehat{T}_M$ are both positive operators on $\dom \widehat{S}_0$. Furthermore, we have for all $h\in \dom \widehat{S}_0$ that 
 \begin{align*}
\|\widehat{T}_mh\|_A^2=\|\widehat{S}_0h\|_A^2+2\sipa{\widehat{S}_0h}{h}+\|\widehat{S}_0\|^2\|h\|_A^2\leq 2\|\widehat{S}_0\|\sipa{\widehat{T}_mh}{h},
\end{align*}
and similarly, 
\begin{align*}
\|\widehat{T}_Mh\|_A^2\leq 2\|\widehat{S}_0\|\sipa{\widehat{T}_Mh}{h}.
\end{align*}
By \cite{KV}*{Theorem 4.2}, there exist two minimal positive extensions $\widehat{A}_m, \widehat{A}_M\in \bha$ of $\widehat{T}_m$ and $\widehat{T}_M$, respectively. Note also that 
\begin{align*}
\|\widehat{A}_Mk\|_A^2\leq 2\|\widehat{S}_0\|\sipa{\widehat{A}_Mk}{k}\qquad\mbox{and}\qquad \|\widehat{A}_Mk\|_A^2\leq 2\|\widehat{S}_0\|\sipa{\widehat{A}_Mk}{k} 
\end{align*}
hold for all $k\in\hila$. Now set $$\widehat{S}_m:=\widehat{A}_m-\|\widehat{S}_0\|\qquad\mbox{and}\qquad \widehat{S}_M:=\|\widehat{S}_0\|-\widehat{A}_M.$$ Clearly,  $\widehat{S}_m,\widehat{S}_M\in\bha$ are both self-adjoint extensions of  $\widehat{S}_0$. For $k\in\hila$, 
 \begin{align*}
 \|\widehat{S}_mk\|_A^2=\|\widehat{T}_mk\|^2_A-2\|\widehat{S}_0\|^2\sipa{\widehat{T}_mk}{k}+\|\widehat{S}_0\|\|k\|^2_A\leq\|\widehat{S}_0\|\|k\|^2_A,
 \end{align*}
 and therefore $\|\widehat{S}_m\|=\|\widehat{S}_0\|=\alpha_A(S_0)$. Similarly,  $\|\widehat{S}_M\|=\alpha_A(S_0)$. Letting  $$S_m:=J\widehat{S}_mJ^*\qquad\mbox{and}\qquad S_M:=J\widehat{S}_MJ^*,$$ we conclude that   $S_m,S_M\in\lef$ are self-adjoint operators such  that 
 \begin{equation*}
\alpha_A(S_m)=\alpha_A(S_M)=\alpha_A(S_0).    
\end{equation*}
 Finally,    for $x\in \dom S_0$ and $y\in E$, 
\begin{align*}
\dual{S_mx}{y}=\sipa{\widehat{S}_m(Ax)}{Ay}=\sipa{\widehat{S}_0(Ax)}{Ay}=\dual{S_0x}{y},
\end{align*} 
hence $S_0\subset S_m$. A similar calculation shows that $S_0\subset S_M$ holds as well.

To prove \eqref{E:intervalSmSM} let $S\in [S_m,S_M]$ and take $x\in\dom S_0$ and $y\in E$. Since $S-S_m\geq0$ it follows that 
\begin{align*}
    \abs{\dual{(S-S_m)x}{y}}^2&\leq \dual{(S-S_m)x}{x}\dual{(S-S_m)y}{y}\\
    &\leq\dual{(S_M-S_m)x}{x}\dual{(S-S_m)y}{y}=0,
\end{align*}
hence $Sx=S_mx=S_0x$, that is, $S_0\subset S$. On the other hand $S_m\leq S\leq S_M$ implies that 
\begin{align*}
\abs{\dual{Sx}{y}}&\leq \dual{(S-S_m)x}{x}^{1/2}\dual{(S-S_m)y}{y}^{1/2}+\abs{\dual{S_mx}{y}} \\
&\leq 3\alpha_A(S_0)\dual{Ax}{x}^{1/2}\dual{Ay}{y}^{1/2},
\end{align*}
hence there is a symmetric operator $\widehat{S}\in\bha$ with $\|\widehat{S}\|\leq 3  \alpha_A(S_0)$ such that $S=J\widehat{S}J^*$. It is clear that $\widehat S_m\leq \widehat{S}\leq \widehat{S}_M$, and thus $\alpha_A(S_0)=\|\widehat S\|=\alpha_A(S)$. Assume conversely that $S_0\subset S$ is any self-adjoint extension such that $\alpha_A(S)=\alpha_A(S_0)$.  Then it is clear that $\alpha_A(S_0)\pm \widehat{S}$ are bounded positive extensions of $\alpha_A(S_0)\pm \widehat{S}_0$, hence $\widehat A_m \leq \alpha_A(S_0) +\widehat{S}$ and $\widehat A_M\leq \alpha_A(S_0)-\widehat{S}$. Consequently, $\widehat{S}_m\leq \widehat{S}\leq \widehat S_M$ and also $S_m\leq S\leq S_M$. The proof is complete.
\end{proof}
In the following corollary we recover the classical result of Krein on self-adjoint norm-preserving extensions.
\begin{corollary}\label{C:Krein}
Let $\hil$ be a Hilbert space and let $S_0:\dom S_0\to\hil$ be a bounded symmetric operator. Then $S$ admits  two self-adjoint norm-preserving extensions $S_m,S_M\in\bh$ such that    the interval $[S_m, S_M]$ consists exactly of all self-adjoint norm-preserving extensions of $S_0:$
\begin{equation*}
    [S_m,S_M]=\set{S\in\bh}{S_0\subset S=S^*, \|S_0\|=\|S\|}.
\end{equation*}
If a self-adjoint operator $B\in\bh$ leaving $\dom S_0$ invariant satisfies $BS_0\subset S_0B$ then 
\begin{equation}\label{E:BSM}
    S_mB=BS_m,\qquad S_MB=BS_M.
\end{equation}
\end{corollary}
\begin{proof}
From the proof of the previous theorem we see that $S_m =(S_0+\|S_0\|)_N-\|S_0\|$ and that $S_M =\|S_0\|-(\|S_0\|-S_0)_N.$ Here we have 
\begin{equation*}
    B(\|S_0\|-S_0)\subset (\|S_0\|-S_0)B \qquad\mbox{and}\qquad B(S_0+\|S_0\|)\subset (S_0+\|S_0\|)B,
\end{equation*}
hence $B(S_0+\|S_0\|)_N=(S_0+\|S_0\|)_NB$ and $B(\|S_0\|-S_0)_N=(\|S_0\|-S_0)_NB$ due to \cite{KV}*{Corollary 4.3}. This clearly gives \eqref{E:BSM}.
\end{proof}
\section{A generalized Strong Parrott Theorem}\label{S:Parrott}

The aim of this section is to generalize Parrott's famous theorem \cite{parrott} on contractive extensions of $2$ by $2$ block operator-valued matrices, which is one of the crucial results in extension and dilation theory. As an application, we will deduce Yamada's recent result \cite{Yamada}*{Theorem 4} on the extension of the Strong Parrott Theorem \cites{F-T,Timotin}.

\begin{theorem}\label{T:dualext}
Let $\dual{F_1}{E_1}^{}_1$ and $\dual{F_2}{E_2}^{}_2$ be two $w^*$-sequentially complete anti-dual pairs and let $T_1:E_1\supseteq \dom T_1\to F_2$ and $T_2:E_2\supseteq \dom T_2\to F_1$  be linear operators such that 
\begin{align*}
\dual{T_1x_1}{x_2}^{}_2=\overline{\dual{T_2x_2}{x_1}^{}_1},\qquad x_1\in\dom T_1, x_2\in\dom T_2.
\end{align*}
Assume furthermore that there exist two positive operators $A_i\in \mathscr{L}(E_i;F_i)$ and constants $\alpha_i\geq 0$, $(i=1,2)$ such that the following estimates hold true: 
\begin{align*}
\abs{\dual{T_1x_1}{y_2}_2}^2&\leq \alpha_1 \dual{A_1x_1}{x_1}_1\dual{A_2y_2}{y_2}_2,\qquad x_1\in\dom T_1, y_2\in E_2,\\
\abs{\dual{T_2x_2}{y_1}_1}^2&\leq \alpha_2 \dual{A_1y_1}{y_1}_1\dual{A_2x_2}{x_2}_2,\qquad x_2\in\dom T_2, y_1\in E_1.
\end{align*}
Then there exists a $T\in\mathscr{L}(E_1;F_2)$ such that $T_1\subseteq T$ and $T_2\subseteq T^*$ and that 
\begin{align*}
\abs{\dual{Ty_1}{y_2}_1}^2&\leq \max\{\alpha_1,\alpha_2\}\cdot \dual{A_1y_1}{y_1}_1\dual{A_2y_2}{y_2}_2,\qquad y_1\in E_1, y_2\in E_2.
\end{align*}
\end{theorem}

\begin{proof}
Let us consider the pair $(E_1\times E_2,F_1\times F_2)$ with antiduality 
$$\sdual{(f_1,f_2)}{(e_1,e_2)}=\dual{f_1}{e_1}_1+\dual{f_2}{e_2}_2,$$
and introduce the following linear operator
\begin{align*}
S_0:E_1\times E_2\supseteq \dom T_1\times\dom T_2\to F_1\times F_2,\qquad S_0(x_1,x_2):=(T_2x_2,T_1x_1).
\end{align*}
An easy calculation shows that $S_0$ is symmetric: indeed, for $x_1,z_1\in\dom T_1$ and $x_2,z_2\in\dom T_2$
\begin{align*}
\sdual{S_0(x_1,x_2)}{(z_1,z_2)}&=\dual{T_2x_2}{z_1}_1+\dual{T_1x_1}{z_2}_2=\overline{\dual{T_1x_z}{z_1}_1}+\overline{\dual{T_2z_2}{x_1}_2}\\
&=\overline{\sdual{S_0(z_1,z_2)}{(x_1,x_2)}}.
\end{align*}
Consider  the positive operator $\Lambda:=\kismatrix{A_1}{0}{0}{A_2}$ acting between $E_1\times E_2$ and $F_1\times F_2$, then 
\begin{multline*}
\abs{\sdual{S_0(x_1,x_2)}{(y_1,y_2)}}\leq \abs{\dual{T_2x_2}{y_1}_1}+\abs{\dual{T_1x_1}{y_2}_2}\\ \leq \alpha\cdot\Big(\dual{A_1x_1}{x_1}_1^{1/2}\dual{A_2y_2}{y_2}_2^{1/2}+\dual{A_1y_1}{y_1}_1^{1/2}\dual{A_2x_2}{x_2}_2^{1/2}\Big)\\
\leq \alpha\cdot \sdual{\Lambda(x_1,x_2)}{(x_1,x_2)}^{1/2}\sdual{\Lambda(y_1,y_2)}{(y_1,y_2))}^{1/2}
\end{multline*}
for $x_i\in \dom T_i, y_i\in E_i$ $(i=1,2)$ and with $\alpha:=\max\{\sqrt{\alpha_1},\sqrt{\alpha_2}\}$. This means that $S_0$ and $\Lambda$ fulfill all the conditions of Theorem \ref{T:symmetricext}. Hence, we conclude that there exists a self-adjoint operator $S\in\mathscr{L}(E_1\times E_2;F_1\times F_2)$ which extends $S_0$ and that 
\begin{align*}
\abs{\sdual{S(y_1,y_2)}{(w_1,w_2)}}^2\leq \alpha^2\cdot\sdual{\Lambda(y_1,y_2)}{(y_1,y_2)}\sdual{\Lambda(w_1,w_2)}{(w_1,w_2)},
\end{align*}
$y_i,w_i\in E_i$ $(i=1,2).$ Let us interpret $S$ as an operator matrix of the form 
\begin{align*}
S=\opmatrix{B_1}{T^*}{T}{B_2},
\end{align*}
where $B_i\in\mathscr{L}(E_i;F_i)$ are self-adjoint  operators and  $T\in \mathscr{L}(E_1;F_2)$.
We claim that $T$ possesses the desired properties. Indeed,
\begin{align*}
\dual{Tx_1}{y_2}_2=\sdual{S(x_1,0)}{(0,y_2)}
                =\sdual{S_0(x_1,0)}{(0,y_2)}=\dual{T_1x_1}{y_2}_2,
\end{align*}
for $x_1\in \dom T_1$, $y_2\in E_2$,  and similarly,
\begin{align*}
\dual{T^*x_2}{y_1}_1=\sdual{S(0,x_2)}{(y_1,0)}
                =\sdual{S_0(0,x_2)}{(y_1,0)}=\dual{T_2x_2}{y_1}_1,
\end{align*}
for $x_2\in \dom T_2$, $y_1\in E_1$, thus we conclude that $T_1\subset T$ and $T_2\subset T^*$. Finally, 
\begin{align*}
\abs{\dual{Ty_1}{y_2}_2}^2&=\abs{\sdual{S(y_1,0)}{(0,y_2)}}^2\\ 
                        &\leq  \alpha^2\cdot\sdual{\Lambda(y_1,0)}{(y_1,0)}\sdual{\Lambda(0,y_2)}{(0,y_2)}\\
                        &= \alpha^2\cdot\dual{A_1y_1}{y_1}_1\dual{A_2y_2}{y_2}_2,
\end{align*}
for $y_i\in E_i, i=1,2$, which completes the proof.
\end{proof}
Using the generalized Parrott theorem above, we obtain a new proof for a recentresult of A. Yamada \cite{Yamada}*{Theorem 4}. 
\begin{corollary}
Let $\dual{F_1}{E_1}^{}_1$, $\dual{F_2}{E_2}^{}_2$ be anti-dual pairs and let $\hil$, $\kil$ be Hilbert spaces. For $S_1\in\mathscr{L}(E_1,\hil)$, $S_2\in\mathscr{L}(E_1,\kil)$, $T_1\in \mathscr{L}(\hil,F_2)$ and $T_2\in\mathscr{L}(\kil,F_2)$ the following conditions are equivalent:
\begin{enumerate}[label=\textup{(\roman*)}]
    \item $T_1S_1=T_2S_2$, $S_2^*S^{}_2\leq  S_1^*S^{}_1$ and $T^{}_1T_1^*\leq T^{}_2T_2^*$,
    \item there exists $X\in\mathscr{B}(\hil,\kil)$, $\|X\|\leq 1$, such that $XS_1=S_2$ and $T_2X=T_1$, i.e., $X$ makes the following diagram commutative:
\begin{center}
\begin{tikzpicture}
\draw (12,4) node (A) {$E_1$};
\draw (10,2) node (B) {$\hil$};
\draw (14,2) node (C) {$\kil$};
\draw (12,0) node (D) {$F_2$};
\draw[ thick,->] (A) -- (B)
 node[pos=0.5,above] {$S_1~$~};
 \draw[ thick,->] (A) -- (C)
 node[pos=0.5,above] {~$\,S_2$};
 \draw[thick, dashed,->] (B) -- (C)
 node[pos=0.5,above,sloped] {$X$};
\draw[thick,->] (B) -- (D)
 node[pos=0.4,below] {$T_1$~};
\draw[thick,->] (C) -- (D)
 node[pos=0.4,below] {~$T_2$};
\end{tikzpicture}
\end{center}
\end{enumerate}
\end{corollary}
\begin{proof}
Implication (ii)$\Rightarrow$(i) is straightforward, so we only prove that (i) implies (ii). Consider the anti dual pairs $\sip \hil\hil$ and $\sip \kil\kil$ and the operators $$X_0:\hil\supseteq \ran S_1\to\kil\qquad\mbox{and}\qquad X_1:\kil\supseteq \ran T_2^*\to\hil,$$ 
defined by 
\begin{equation*}
    X_0(S_1x_1)\coloneqq S_2x_1, \qquad X_1(T_2^*x_2)\coloneqq T_1^*x_2.
\end{equation*}
From (i) we see that $X_0,X_1$ are well defined contractions such that 
\begin{align*}
        \sip{X_0(S_1x_1)}{T_2^*x_2}=\sip{S_2x_1}{T_2^*x_2}=\sip{S_1x_1}{T_1^*x_2} =\sip{S_1x_1}{X_1(T_2^*x_2)},
\end{align*}
for every $x_1\in E_1$ and $x_2\in E_2$.
Hence the pair $X_0, X_1$ fulfills the conditions of Theorem \ref{T:dualext} with $A_1=I_\hil$, $A_2=I_\kil$ and $\alpha_1=\alpha_2=1$. Consequently, there exists  $X\in\mathscr{B}(\hil,\kil)$, $\|X\|\leq1$ such that $X_0\subset X$, $X_1\subset X^*$, and therefore $XS_1=S_2$ and $X^*T_2^*=T_1^*$ which yields (ii).    
\end{proof}
Yamada's work itself generalizes Parrott's theorem \cite{parrott} and the Strong Parrott Theorem  (see \cites{Ando-Hara,Bakonyi, F-T}). So we get the following classical result automatically. 

\begin{corollary}
Let $\hil$ and $\kil$ be Hilbert spaces, $\hil_1\subseteq\hil$ and $\kil_1\subseteq\kil$ be closed linear subspaces, and denote by $P_{K_1}$ the orthogonal projection onto $K_1$. For given contractions $T_1:\hil_1\to\kil$ and $T_1':\hil\to\kil_1$ the following conditions are equivalent:
\begin{enumerate}[label=\textup{(\roman*)}]
\item $P_{K_1}T_1=T_1'|_{H_1},$
\item there exists a contraction $T\in\bhk$ such that 
\begin{align*} T_1=T|_{H_1}\quad\mbox{and}\quad T_1'=P_{K_1}T.
\end{align*}
\end{enumerate}
\end{corollary}
\section{Hermitian extensions of linear functionals}\label{S: funct}

Positive functionals play an important role in the representation theory of algebras. Extension of such functionals has been investigated in many different settings. For example, if $f$ is a positive linear functional defined on a closed ideal in a $C^{*}$-algebra, then $f$ always admits an extension with the same norm (see \cite{Blackadar}*{II.6.4.16}). Positive functionals defined on left-ideals of the full operator algebra possessing normal extension were characterized in \cite{STTfunk}, while positive extendibility of positive functionals defined on left ideals of general ${}^*$-algebras was studied in \cite{KV}. The aim of this section is to demonstrate how our anti-dual pair setting can be used to construct hermitian extensions of linear functionals in the unital ${}^*$-algebra setting.
\par
Let $\alg$ be a unital *-algebra with unit $1$. If $\ideal\subseteq \alg$ is a left ideal then we call a linear functional $g:\ideal\to\dupC$ is  \emph{symmetric}, if 
\begin{equation*}
    g(b^*a)=\overline{g(a^*b)},\qquad a,b\in\ideal.
\end{equation*}
A linear functional $g\in\alg^*$ is called \emph{hermitian} if it satisfies 
\begin{equation*}
    g(x^*)=\overline{g(x)},\qquad x\in\alg.
\end{equation*}
It is easy to check that  a linear functional $g\in\alg^*$ is hermitian if and only if it is symmetric. Note that for a *-algebra without unit element this equivalence does not longer hold.

Assume that a positive linear functional $f:\alg\to\dupC$ is given. We say that the symmetric functional $g:\ideal\to\dupC$ is $f$-\emph{bounded}, if 
\begin{equation}\label{E:f-bounded}
    \abs{g(x^*a)}^2\leq \alpha^2 f(x^*x)f(a^*a),\qquad x\in\alg,a\in\ideal,
\end{equation}
holds for some $\alpha>0$. The $f$-bound $\alpha_f(g)$ is defined as the smallest constant $\alpha$  that fulfills \eqref{E:f-bounded}. If $\ell:\ideal\to\dupC$ is a linear functional then the correspondence 
\begin{equation*}
    \dual{La}{x}\coloneqq \ell(x^*a),\qquad a\in\ideal, x\in\alg,
\end{equation*} 
defines a linear operator $L:\ideal\to\bar{\alg}^*$. Clearly, $L$ is positive if $\ell$ is positive and $L$ is symmetric if $\ell$ is so. Suppose now that $f\in\alg^*$ is  a positive functional and denote by $A$ the positive operator associated with $f$, i.e.,  
\begin{equation*}
    \dual{Ax}{y}=f(y^*x),\qquad x,y\in\alg.
\end{equation*}
Let $\hila$ denote the corresponding Hilbert space, that is, $\hila$ is the completion of $\ran A$ endowed with the inner product \eqref{E:sipa}. Observe that in that case we have
\begin{equation*}
    \sipa{Ax}{Ay}=f(y^*x),\qquad x,y\in\alg.
\end{equation*}
Consider the canonical embedding $J:\hila\to\bar{\alg}^*$ in \eqref{E:J0} and recall its useful properties 
\begin{equation*}
    J^*x=Ax,\qquad x\in \alg,
\end{equation*}
and $JJ^*=A$. 
Assume in addition that 
\begin{equation}\label{E:admissible}
    \abs{f(y^*x^*xy)}\leq M_x f(y^*y),\qquad x,y\in\alg, 
\end{equation}
for some $M_x\geq0$. This assures that the operators $\pi_f(x)$ defined by
\begin{equation*}
    \pi_f(x)(J^*y)\coloneqq J^*(xy),\qquad y\in\alg,
\end{equation*}
are continuous on $\hila$ by norm bound $M_x^{1/2}$. Thus, for every $x\in\alg$ we can extend $\pi_f(x)$ to an element of $\bha$. It is then immediate that $\pi_f:\alg\to\bha$ is a $^*$-homomorphism such that 
\begin{equation}\label{E:pif}
    f(x)=\sipa{\pi_f(x)(J^*1)}{J^*1},\qquad x\in\alg.
\end{equation}
A positive functional satisfying \eqref{E:admissible} (and hence \eqref{E:pif}) will be called \emph{representable} \cite{Sebestyen84}.

\begin{theorem}\label{T:hermitianext}
Let $\alg$ be a unital $^*$-algebra, $\ideal\subseteq \alg$ a left ideal and $f\in\alg^*$ a representable positive functional. If  $g^{}_0:\ideal\to\dupC$ is an $f$-bounded symmetric functional with $f$-bound $\alpha_f(g^{}_0)$ then there exist two distinguished $f$-bounded hermitian functionals $g^{}_m,g^{}_M\in\alg^*$ with $f$-bound $\alpha_f(g_m)=\alpha_f(g^{}_M)=\alpha_f(g_0)$ extending $g^{}_0$. Furthermore, $g_m\leq g^{}_M$ and the interval $[g_m,g_M]$ consists of all hermitian $f$-bound preserving extensions   of $g^{}_0$:
\begin{equation*}
    [g_m,g_M]=\set{g\in\alg^*}{g_0\subset g=g^*, \alpha_f(g)=\alpha_f(g_0)}.
\end{equation*}
\end{theorem}
\begin{proof}
Along the lines of the proof of Theorem \ref{T:symmetricext}, let us introduce a symmetric operator $S^{}_0$ on $\dom S_0\coloneqq \set{J^*a}{a\in\ideal}$ such that 
\begin{equation*}
    \sipa{S_0(J^*a)}{J^*x}=g^{}_0(x^*a),\qquad a\in\ideal,x\in\alg.
\end{equation*}
 A straightforward calculation shows that $S_0:\dom S_0\to\hila$ is bounded with norm $\|S_0\|=\alpha_f(g_0)$. For $x,y\in \alg$ and $a\in\ideal$ we have 
\begin{align*}
\sipa{S_0\pi_f(x)(J^*a)}{J^*y}&=\sipa{S_0J^*(xa)}{J^*y}=g(y^*xa)=\sipa{S_0(J^*a)}{J^*(x^*y)}\\
                          &=\sipa{S_0J^*a}{\pi_f(x^*)J^*y}=\sipa{\pi_f(x)S_0(J^*a)}{J^*y},
\end{align*}
hence we infer that  $\dom S_0$ is $\pi_f$-invariant and  $\pi_f(x)S_0\subset S_0\pi_f(x)$ for all $x\in\alg$. By Corollary \ref{C:Krein}, it follows that there exist two norm preserving self-adjoint extensions $S_m,S_M\in\bha$ of $S_0$ such that 
\begin{equation}\label{E:pi(x)S}
    \pi_f(x)S_m=S_m\pi_f(x),\qquad\mbox{and}\qquad \pi_f(x)S_M=S_M\pi_f(x),
\end{equation}
whenever $x$ is self-adjoint, and hence also for every $x\in\alg$. We claim that the functionals
\begin{equation*}
    g^{}_m(x)\coloneqq \sipa{S_mJ^*x}{J^*1},\qquad\mbox{and}\qquad g^{}_M(x)\coloneqq \sipa{S_mJ^*x}{J^*1},\qquad x\in\alg,
\end{equation*}
fulfill all conditions of the statement. First we observe that $g_m$ and $g_M$ are hermitian: indeed, by \eqref{E:pi(x)S} we have for every $x\in\alg$ that   
\begin{align*}
    g^{}_m(x^*)&=\sipa{S_mJ^*x^*}{J^*1}=\sipa{\pi^{}_f(x^*)S_mJ^*1}{J^*1}=\sipa{S_mJ^*1}{J^*x}=\overline{g_m(x)}
\end{align*}
A similar argument shows that $g_M$ is hermitian. Next, observe that $g_m$ and $g_M$ extend $g_0$, because for every $a\in\ideal$ 
\begin{equation*}
    g_m(a)=\sipa{S_mJ^*a}{J^*1}=\sipa{S_0J^*a}{J^*1}=g(a)
\end{equation*}
holds, and similarly, $g^{}_M(a)=g(a)$. Finally, we have 
\begin{equation*}
\|S_m\|=\|S_M\|=\|S\|=\alpha_f(g_0),    
\end{equation*}
whence it follows readily that $g_m$ and $g_M$ have $f$-bound $\alpha_f(g_0)$.
 
Let now $g\in\alg^*$ be an arbitrary hermitian extension of $g_0$ having $f$-bound $\alpha_f(g_0)$. Then there is a self-adjoint operator $S\in\bha$, $\|S\|=\alpha_f(g_0)$ such that 
\begin{equation*}
    \sipa{SJ^*x}{J^*y}=g(y^*x),\qquad x,y\in\alg.
\end{equation*}
It is clear that $S_0\subset S$ and therefore $S_m\leq S\leq S_M$, due to Corollary \ref{C:Krein}. As a straightforward consequence we conclude that  $g_m\leq g\leq g_M$. Suppose conversely that $g\in\alg^*$ is a hermitian functional such that $g_m\leq g\leq g_M$. First observe that $g$ is $f$-bounded as it satisfies 
\begin{align*}
\abs{g(y^*x)}&\leq (g-g_m)(x^*x)^{1/2}(g-g_m)(y^*y)^{1/2}+\abs{g_m(y^*x)} \\
&\leq 3\alpha_f(g^{}_0)f(x^*x)^{1/2}f(y^*y)^{1/2}.
\end{align*}
Hence there exists a self-adjoint operator $S\in\bha$, such that 
\begin{equation*}
    g(y^*x)=\sipa{JSJ^*x}{y},\qquad x,y\in\alg.
\end{equation*}
From $g_m\leq g\leq g_M$ it follows  that $S_m\leq S\leq S_M$,  and therefore  $S_0\subset S$ by Corollary \ref{C:Krein}. Consequently,
\begin{equation*}
    g^{}_0(a)=\sipa{S_0J^*a}{J^*1}=\sipa{SJ^*a}{J^*1}=g(a),\qquad a\in\ideal,
\end{equation*}
thus $g_0\subset g$. The proof is complete.
\end{proof}
We remark that Theorem \ref{T:hermitianext} provides only a sufficient condition for the existence of hermitian extensions. On $C^*$-algebras, the statement of Theorem \ref{T:hermitianext} may be improved in two ways: first, the condition on $f$ of being representable may be replaced by the formally weaker one of being positive. On the other hand,  the existence of a dominating positive functional is both necessary and sufficient. 
\begin{corollary}
Let $\alg$ be a unital $C^*$-algebra and $\ideal\subseteq\alg$ a left ideal. A linear functional $g^{}_0:\ideal\to\dupC$ possesses a continuous hermitian extension $g$ if and only if $g_0$ is symmetric and $f$-bounded for some positive functional $f\in\alg^*$. 
\end{corollary}
\begin{proof}
On a $C^*$-algebra, every positive functional $f$ is representable. If $g^{}_0$ is symmetric and $f$-bounded then $g^{}_0$ has a hermitian extension $g$ that is of the form 
\begin{equation*}
    g(a)=\sip{SJ^*a}{J^*1},\qquad a\in\alg,
\end{equation*}
where $S$ is a bounded self-adjoint operator on $\hila$. It follows therefore that $g$ is continuous. For the converse, assume that $g\in\alg^*$ is a continuous hermitian extension of $g^{}_0$. Let $g=g^{}_+-g^{}_-$ be the Hahn--Jordan decomposition of $g$, with $g^{}_+,g^{}_-\in\alg^*$ positive functionals (see \cite{K-R}). Letting $f\coloneqq g^{}_++g^{}_-$, it is easy to check that $g_0$ is $f$-bounded with bound $\alpha_f(g_0)=4$.
\end{proof}
\textbf{Acknowledgement.} We thank the Referees for their time and for careful reading of the manuscript. Their comments helped improving the quality of the paper immensely.


\end{document}